\documentclass[journal,twoside,web]{ieeecolor}
\usepackage{generic}
\usepackage{cite}
\usepackage{amsmath,amssymb,amsfonts}
\usepackage{algorithmic}
\usepackage{graphicx}
\usepackage{textcomp}
\usepackage{float}
\usepackage{pdfpages}
\usepackage{soul}
\usepackage{nicefrac}

\usepackage{babel,blindtext}
\usepackage{savesym}
\savesymbol{checkmark}
\usepackage{dingbat}
\usepackage{mathtools}  
\usepackage{commath}
\usepackage{array}
\usepackage{derivative}
\usepackage{enumerate}
\usepackage{MnSymbol}
\usepackage{cuted}
\usepackage{dblfloatfix}
\usepackage{subcaption}

\newtheorem{theorem}{Theorem}
\newtheorem{lemma}{Lemma}

\newtheorem{remark}{Remark}
\newtheorem{assumption}{Assumption}

\usepackage{hyperref}
\hypersetup{colorlinks,allcolors=black}

\usepackage[capitalize,nameinlink]{cleveref}

\usepackage{stackengine}
\newcommand\barbelow[1]{\stackunder[1.2pt]{$#1$}{\rule{.8ex}{.075ex}}}

\begin{document}
\title{Approximation-free  control for unknown systems with performance and input constraints}
\author{Pankaj K Mishra and Pushpak Jagtap, \IEEEmembership{Member, IEEE}
\thanks{This paper is submitted for consideration of publication.}
\thanks{This work was partly supported by the Google Research Grant, the SERB Start-up Research Grant, and the CSR Grant by Nokia on Networked Robotics.}
\thanks{The authors are with the Robert Bosch Centre for Cyber-Physical Systems,
Indian Institute of Science, Bangalore 560012, India (e-mail:
\{pankajmishra,pushpak\}@iisc.ac.in)}}

\makeatletter
\AddToShipoutPicture{%
            \setlength{\@tempdimb}{.5\paperwidth}%
            \setlength{\@tempdimc}{.5\paperheight}%
            \setlength{\unitlength}{1pt}%
            \put(\strip@pt\@tempdimb,\strip@pt\@tempdimc){%
        \makebox(0,0){\rotatebox{45}{\textcolor[gray]{0.92}%
        {\fontsize{6cm}{6cm}\selectfont{Preprint}}}}%
            }%
}

\maketitle
\thispagestyle{empty} 

\begin{abstract}
This paper addresses the problem of tracking control for an unknown nonlinear system with time-varying bounded disturbance subjected to prescribed Performance and Input Constraints (PIC). Since simultaneous prescription of PIC involves a trade-off, we propose an analytical feasibility condition to prescribe feasible PIC which also yields feasible initial state space as corollary results. Additionally, an approximation-free controller is proposed to guarantee that the tracking performance adheres to the prescribed PIC. The effectiveness of the proposed approach is demonstrated through numerical examples.
\end{abstract}


\section{Introduction}
In control system design, constraints play a crucial role in ensuring the stability, safety, and performance of the system. Constraints can be imposed on various system parameters, such as the input, state, and performance (output and error) to limit their values within a certain range. Among these constraints, input and performance constraints are the most common and have been extensively studied in the literature. 

Performance constraint aims for lower steady-state error, safe transient response, and fast convergence of tracking error. In contrast, input constraint focuses on actuator safety or control effort minimization. Improving performance with limited resources is always difficult. Same with simultaneous prescription of performance and input constraints (PIC) \cite{Yong2020}. In recent years, significant progress has been made in designing controllers for systems subject to PIC. Model Predictive Control (MPC), a popular optimal control technique, provides a powerful framework for incorporating various forms of constraints. However, it requires knowledge of the system and solves an optimization problem in each control loop, making it computationally expensive.\par
To deal with performance constraints other than MPC, important techniques are Barrier Lyapunov Function (BLF)\cite{TEE2009918,TEE20112511,5499019, 9621216} and error transformation \cite{Bechlioulis2008a,Bechlioulis2014a} based control design. The advantage of \cite{TEE2009918,TEE20112511,5499019,Bechlioulis2008a,Bechlioulis2014a} are that it is a Lyapunov-based controller design technique, thus providing a suitable framework to deal with the unknown system and also to deal with various class of problem such as fixed-time stabilization \cite{9707600}, leader–follower consensus for multi-agent systems \cite{9007508}, chasing unknown target \cite{9694521}, tracking control of interconnected non-affine nonlinear systems\cite{Mishra2021f}, and event trigger control \cite {8640807}. To conclude, it is worth noting that a considerable amount of research has been conducted in the field of prescribed performance constraints. Similarly, significant contributions have been made to the controller design for nonlinear systems with input constraints. Some noteworthy works in this area include \cite{HU2008512, Wen2011, Chen2015, CHEN2011452}. \par

Despite the extensive research on performance and input constraints  separately, the PIC problem remains challenging due to their inevitable trade-off nature in simultaneous prescription. Few results are available in the literature addressing both forms of constraints\cite{berger2022input,Hopfe2010,Hopfe2010a,Kanakis2020}, and those that do exist typically rely on  ad-hoc assumptions. In  \cite{berger2022input, Hopfe2010,Hopfe2010a} authors relax the performance constraint  whenever the input saturation is active. In \cite{Kanakis2020} assumptions are made on the existence of a feasible set of control inputs for a given initial conditions and actuator saturation limit. The major issue in this direction is the simultaneous prescription of feasible constraints. It is impractical to guarantee the feasibility of  arbitrarily prescribed constraints.\par

 Many practical systems always operate in some specified regions where they are controllable under input constraints \cite{Yong2020}. In the presence of input constraints, one cannot stabilise the unstable system globally. There is always a feasible set of initial conditions for input constraint. Also, given any desired trajectory for an uncertain nonlinear system with unknown bounded disturbances and arbitrary input constraints, it is certainly impractical to guarantee that the desired trajectory is trackable. For instance, a large external disturbance or a desired trajectory with a very large upper bound will invariably require the same level of opposing control command, which may extend beyond the input constraint \cite{Asl2019,10093463,10004950}, contradicting the arbitrary prescription of input constraint. Also, given any input constraint, the arbitrary prescription of performance constraint is impractical and always necessitates the need to devise a feasibility condition for a simultaneous prescription. Motivated by this, the contributions of our work are as follows.
 \begin{enumerate}
     \item Feasibility condition has been devised to avoid  arbitrary prescription of PIC.
    \item Utilizing a backstepping control framework, a novel approximation-free controller is proposed for tracking control in the presence of PIC
    \item Using the feasibility condition,  a viable set of  initial conditions is computed for the proposed controller. This contribution is a corollary of the first. 
 \end{enumerate}
The paper is structured as follows. Section \ref{sec2} provides an overview of the preliminary concepts and formulates the problem. Section \ref{sec3} presents the proposed controller. The mathematical analysis of the proposed controller is discussed in Section \ref{sec:anlysis}. The effectiveness of the proposed strategy is demonstrated through simulation studies in Section \ref{sec:simu}. Finally, Section \ref{sec:concl} concludes the paper.\par
\textbf{Notations:} The set of real, positive real, nonnegative real, and positive integer numbers are denoted by $\mathbb{R}$, $\mathbb{R}^+$, $\mathbb{R}_0^+ $, and $\mathbb{N}$, respectively. $\mathbb{N}_n$: $\{1,\ldots , n\}$, $n$ is positive integer. $\mathcal{L}^\infty$ represents the set of all essentially bounded measurable functions. For $x (t)\in \mathbb{R}$, $x\uparrow a$: $x$ approaches a real value $a$ from the left side,  $x\downarrow a$:  $x$ approaches a real value $a$ from the right side. $\bar 0_i$: $i$-dimensional zero vector.

\section{Preliminaries and the Problem Statement}\label{sec2}
Consider a pure feedback $n$th order nonlinear system
\begin{equation}\label{sys1}
\begin{split}
\dot \xi_i&=f_i\left(\bar{\xi}_i\right)+g_i\left(\bar{\xi}_i\right)\xi_{i+1}+d_i, \forall i\in\mathbb{N}_{n-1}\\
\dot \xi_n&=f_i\left(\bar{\xi}_n\right)+g_n\left(\bar{\xi}_n\right)\upsilon+d_n,\\
y&=\xi_1,
\end{split}
\end{equation} 
where for $i\in \mathbb{N}_n$, $\bar \xi_i(t)=[\xi_1(t), \ldots, \xi_i(t)]^T \in \mathbb{R}^i$ with $\xi_i(t) \in \mathbb{R}$ is the state vector,  $f_i:\mathbb{R}^{i}\rightarrow\mathbb{R}$ is the unknown smooth nonlinear map with $f_i(\bar{0}_i)=0$, $g_i:\mathbb{R}^{i}\rightarrow\mathbb{R}$ is the unknown control coefficient, and $d_i(t) \in \mathbb{R}$ is the unknown piecewise continuous bounded disturbance. $\upsilon(t) \in \mathsf{U}\subseteq \mathbb{R}$ and $y(t) \in \mathbb{R}$ are the input and output of the system. \par
In order to define the control goal, a few variables are needed, which are defined as follows.
The desired output is defined as $y_d(t) \in \mathbb{R},$ and the output tracking error is defined as $z_1\coloneqq\xi_1-y_d$. In this paper, we will be using a well-known nonlinear control techniques framework, i.e., Backstepping, which will consist of $``n"$ error variables defined as follows,
\begin{align}
    z_i=\xi_i-\upsilon_{i-1}, \forall i\in\mathbb{N}_{n}, \label{error}
\end{align}
where $\upsilon^{}_0=y_d,$ and $z_i,~ \forall i\in\mathbb{N}_{n}$ are error variables of which $z_1$ is output tracking error;  $\upsilon_i,~ \forall i \in \mathbb{N}_{n-1}$ are virtual control inputs. To simplify the mathematical presentation in the paper, we will represent actual input $``\upsilon"$ as  $\upsilon_n=\upsilon$. For these variables, the associated constraints are defined as follows.\par
The performance constraints on the error variable $z_i$ are represented using a time-varying function $\psi_i: \mathbb{R}^+_0\rightarrow \mathbb{R}$, defined as
\begin{align}
    \psi_i(t)=(p_{i}-q_i)e^{-\mu_i t}+q_i, ~ \forall t\geq0, \label{constr} 
\end{align}
where $q_i$, $p_i\ge q_i $ and $\mu_i,~\forall i \in \mathbb{N}_n,$ are arbitrarily chosen positive constants that drive the bounds on the steady-state value, the initial value, and the decay rate of the error variables, respectively. Further, the input constraints on the virtual control inputs and actual input are represented using non-zero positive constants $\bar \upsilon_i \in  \mathsf{U}, ~ \forall i \in \mathbb{N}_{n-1},$  and $\bar \upsilon_n \in  \mathsf{U},$ respectively. It is worth to note that in \eqref{constr}, $\psi_i$ and $\dot \psi_i$ are bounded for $\forall t\geq0,$ and its bounds are given as
\begin{align}
    q_i\le{\psi_i}(t)\le p_i,~  &\forall (t,i)\in \mathbb{R}^+_0 \times \mathbb{N}_n, \label{psib}\\
    \mu_i(q_i - p_i)\le\dot \psi_i (t)\le 0,~ &\forall (t,i)\in \mathbb{R}^+_0 \times \mathbb{N}_n. \label{psidb}
\end{align}

\textit{Control Goal:} It is defined in twofold, first is to devise a feasibility condition to overcome the issue of arbitrary simultaneous prescription of performance and input constraints,  and second is to design a control law for \eqref{sys1} such that: $(i)$ the output $y(t)=\xi_1(t)$ tracks the desired output $y_d(t)$  without violating the prescribed performance constraint $\psi_1(t)$, i.e., output tracking error should satisfy its constraints $|z_1(t)|<\psi_1(t), \forall t\geq0$, $(ii)$ input should follow its prescribed constraints $\bar \upsilon_n$, i.e., $|\upsilon_n(t)|<\bar \upsilon_n, \forall t\geq0$   and $(iii)$ all closed-loop signals are bounded.\par
 In addition to the two goals listed above, this study also contributes by computing the viable set of state initial conditions, which is the third contribution. How to calculate this value is demonstrated in section \ref{sec:simu} (Simulation results and discussion).
\begin{remark}
    Computing a feasible set of the decay rate and bounds of the transient and steady state of the tracking error are other things that could be done using feasibility conditions. But in this paper, we have restricted our contributions to the computation of  feasible initial conditions of the state.
\end{remark}

To achieve the above goal, we raise the following assumptions.
\begin{assumption}[\cite{Zhang2017a}]\label{a1}
For all $i \in \mathbb{N}_n$, there exists a constant $k_i\ge0$  such that 
 $\abs{f_i(\bar\xi_{i})}\le k_i\norm{\bar\xi_{i}}$, where $k_i$ is a known Lipschitz constant.
\end{assumption}
Note that one can use Lipschitz constant inference approaches proposed in \cite{wood1996estimation,Bubeck2011,malherbe2017global} to estimate the Lipschitz constant of the unknown dynamics from a finite number of data collected from the system.
\begin{assumption}\label{ag}
There exist a nonnegative known constants $\barbelow g_i$ and $\bar g_i$, such that $\barbelow g_i\le g_i(\bar \xi_i)\le\bar g_i$, $\forall i \in \mathbb{N}_n.$  
\end{assumption}
\begin{assumption}\label{a3}
The desired trajectory $y_d$ and its time derivative $\dot y_d$ are continuous real-valued functions and there exist positive constants $\bar v_0$ and $r_0$ such that $|y_d(t)|\le \bar v_0$ and $|\dot y_d(t)|\le r_0$, $\forall t \in \mathbb{R}^+_0$.
\end{assumption}
\begin{assumption}\label{a4}
There exists a known constant $\bar d_i\ge 0$ such that disturbances $|d_i(t)|\le \bar d_i$, $\forall i \in \mathbb{N}_n$.
\end{assumption}
\begin{remark}\label{whyassump}
We will see later that assumptions \ref{a1} and \ref{ag} have been mentioned to devise feasibility conditions rather than the design of control.
\end{remark}
\section{Controller Design}\label{sec3}

Before discussing the design of the controller, we first define an auxiliary variable $\theta_i$ as
\begin{align}
    \theta_i=\frac{z_i}{\psi_i},~\forall i \in \mathbb{N}_{n}.\label{theta}
\end{align}


This auxiliary variable will be utilized in our approach for designing an approximation-free controller. In conventional approaches, controller design for unknown systems typically involves adaptive laws or learning agents to estimate unknown parameters in the control law. However, our proposed approach takes a different approach by utilizing the introduced auxiliary variable. We will now discuss the philosophical aspects of the proposed approximation-free controller before presenting it in subsequent sections.\par
\textit{Note:} In several instances in this paper, we will simplify the notation by omitting the explicit indication of the time variable ``$t$". This approach is taken to streamline the presentation without compromising the integrity or validity of our results.

\subsection{Philosophy behind an approximation-free controller}\label{sec:philo}
For a set $\bar{\mathbb{X}}=(-1 ~ 1)$, it can be inferred that in  \eqref{theta}, if $\theta_i(t) \in \bar{\mathbb{X}}$ then $|z_i(t)|<\psi_i(t)$ for all $(t, i) \in { \mathbb{R}_0^+ \times \mathbb{N}_n}$. Thus  if we can design a controller such that $\theta_1(t) \in \bar{\mathbb{X}},$  $\forall t \in \mathbb{R}_0^+ $, then we can assure that  performance constraint goal, i.e., $|z_1|<\psi_1,$ will be achieved. Further, if the designed controller has an additional feature, i.e., if $\theta_1(t) \in \bar{\mathbb{X}},$  $\forall t \in \mathbb{R}_0^+ $  then control input is bounded and satisfy its input constraint, i.e., $|\upsilon_n(t)|<\bar \upsilon_n$ holds  $\forall t \in \mathbb{R}_0^+.$ Based on the aforementioned discussion, the necessary property that the controller must possess can be highlighted as follows. \par
The controller constructed using $\theta_i$, i.e., $\upsilon_i(\theta_i)$, $i \in \mathbb{N}_n,$ should hold following properties,\\
 $(P1)$  $\phi_i=\odv{\upsilon_i}{\theta_i}: \bar {\mathbb{X}} \rightarrow \Omega_i \subset  \mathbb{R},$ $\forall i \in \mathbb{N}_{n}$,\\
 $(P2)$ $\phi_i=\odv{\upsilon_i}{\theta_i}<0~$ $\forall (\theta_i, i) \in \bar{\mathbb{X}} \times \mathbb{N}_{n}$, and\\
 $(P3)$ $\upsilon_i:  \bar{\mathbb{X}} \rightarrow (-\bar v_i~ \bar v_i)$.
\begin{remark}
    Noting the sign of $g_i{(\bar \xi_i)}$ in assumption \ref{ag}, it can be inferred that if $\theta_i$ will approach boundary points of $\bar{\mathbb{X}}$ in either direction, i.e.,  $1$ or $-1$ then properties $P1$ and $P2$, assure that  an equivalent amount of control input  will be applied in the opposite direction, so that  $\theta_i$ never approach its boundary points. Consequently, as discussed previously, this feature of the controller will help in achieving performance constraints. Further, it is easy to infer that  property $P3$ will help achieve input constraint.
\end{remark}
In the next section, we will see that  the proposed controller possesses the above properties.
 \subsection{Proposed approximation-free controller}
Following the properties discussed in section \ref{sec:philo},  with $c_i>0$, the inputs are designed as 
\begin{align}
    \upsilon_i=-\frac{2\bar v_i}{\pi}\arctan{\left(\frac{\pi}{2c_i}\tan{\frac{\pi}{2}\theta_i}\right)},~ \forall i \in \mathbb{N}_n. \label{vi}
\end{align}
Taking the time derivative of \eqref{vi}, we have
\begin{align}
   \dot \upsilon_i&=\phi_i \dot \theta_i,\label{inpd}
\end{align}
where
\vspace{-0.2cm}
\begin{align}
    \dot\theta_i&=\frac{\dot z_i \psi_i -z_i\dot \psi_i}{\psi^2_i}, ~\text{and}\label{thetad}\\
       \phi_i&=\odv{\upsilon_i}{\theta_i}=-\frac{2\pi \bar v_i c_i}{(4c_i^2-\pi^2) \cos^2{(\frac{\pi}{2}\theta_i})+\pi^2}, ~ \forall i \in \mathbb{N}_n,\label{phidr}
\end{align}
obtained by taking time derivatives of \eqref{theta} and \eqref{vi}, respectively.  
From \eqref{phidr} it can be inferred that for all $i \in \mathbb{N}_n$
\begin{align}
   \phi_i\in \Omega_i \subset \mathbb{R}, \label{vidrange}
\end{align}
where $\Omega_i =\left(\barbelow{\phi}_i ~\bar{\phi}_i\right]$ and
\begin{equation}
    \barbelow\phi_i=\begin{cases}
                  -\frac{\pi\bar v_i}{2c_i},~~~&0<c_i<\frac{\pi}{2},\\
                    -\frac{2}{\pi}\bar v_ic_i,~~~&c_i\ge\frac{\pi}{2}.
                    \end{cases} \label{ran1}
\end{equation}
\begin{equation}
    \bar\phi_i=\begin{cases}
                -\frac{2}{\pi}\bar v_ic_i,~~~&0<c_i<\frac{\pi}{2},\\
                  -\frac{\pi\bar v_i}{2c_i},~~~&c_i\ge\frac{\pi}{2}.
                    \end{cases}\label{ran2}
\end{equation}
Following \eqref{vidrange}-\eqref{ran2} and the facts that $\bar v_i>0$ and $c_i>0$, it can be inferred that 
\begin{align}
   \phi_i= \odv{\upsilon_i}{\theta_i}<0, ~ \forall i \in \mathbb{N}_n. \label{deriv}
\end{align}
Recalling \eqref{vi}, and following the fact, $\lim_{\theta_i\uparrow 1}{\upsilon_i}=-\bar v_i(\theta_i)$ and $\lim_{\theta_i\downarrow -1}{\upsilon_i}=\bar v_i(\theta_i), i \in \mathbb{N}_n,$ it can be concluded that
\begin{align}
   \abs{\upsilon_i}&<  \bar v_i, ~\forall (\theta_i,i) \in \bar{\mathbb{X}}\times \mathbb{N}_n.  \label{vib}
  \end{align} 
From \eqref{vidrange}, \eqref{deriv}. and  \eqref{vib}, it can be verified that \eqref{vi} holds  $(P1)$, $(P2)$ and  $(P3)$, respectively.
\begin{remark}\label{remarkparam}
In \eqref{vi}, $\theta_i=\frac{z_i}{\psi_i}$ and $\psi_i=(p_{i}-q_i)e^{-\mu_i t}+q_i$, so for each $i \in \mathbb{N}_n$, the designed inputs \eqref{vi} involves the fictitious design parameter $\bar v_i, c_i, p_i, q_i$, and $\mu_i$. However, for $i=1$ and $i=n$, ($p_1, q_1$, $\mu_1$) and ($\bar v_n =\bar  \upsilon$) will be prescribed beforehand as a performance constraint parameter for output tracking and control input bounds, respectively.\par
\end{remark}
The main results of the paper are theorized below.
\begin{theorem}\label{thm}
    Consider the system \eqref{sys1} satisfying Assumption \eqref{a1}-\eqref{a4}, if virtual control and control inputs are designed using \eqref{vi}, and the performance constraint $\psi_1$ and input constraint $\bar \upsilon$ is prescribed such that the following feasibility conditions hold:
\begin{equation}\label{cond}
    \begin{split}
        \varphi_i&<(\bar g_i+\barbelow g_i)\bar v_i+\mu_i(q^{}_{i}-p^{}_{i}), ~ \forall i \in \mathbb{N}_n,
    \end{split}
\end{equation}
 where 
 \begin{align}
     \varphi_i&= k_i\norm{\delta_{i}}+\bar d_i+\bar g_i p_{i+1}+\bar g_i \bar v_i+r_{i-1},~\forall i \in \mathbb{N}_{n-1} \label{varphi}\\
     \varphi_n&= k_n\norm{\delta_{n}}+\bar d_n+\bar g_n \bar v_n+r_{n-1}\label{varphin},
 \end{align} 
 with $\delta_{i}=[p^{}_1+\bar v_0, \ldots,p_i+\bar v_{i-1}]^T$ and 
\begin{align}\label{ri}
        r_i=\left(\frac{\varphi_i}{q_i}+\frac{\mu_i(p_i-q_i)}{p_i} \right)|\barbelow \phi_i|,~\forall i \in \mathbb{N}_{n},
\end{align} and a trivial condition 
\begin{align}\label{trivcond}
    |z_i(0)|<\psi_i(0), \text{~or~} |z_i(0)|<p_i,
\end{align}
then the system output and input will never transgress its prescribed performance constraint, i.e.,  $|z_1(t)|<\psi_1(t)$ and $|\upsilon(t)|<\bar \upsilon$, $\forall t \in \mathbb{R}^+_0$, respectively, and all the closed-loop signals will remain bounded.
\end{theorem}
\begin{proof}
    The proof is given in Section \ref{sec:stability analysis}.
\end{proof}
\begin{remark}\label{vidotbound}
    In the preliminary section of the proof, i.e., Section \ref{sec:preanalysis}, we will find  that $r_i$, mentioned in  \eqref{ri} of Theorem \ref{thm}, is actually the bounds of virtual input derivatives (See Lemma \ref{lemmar}).
\end{remark}
\begin{remark}
    In brief, Theorem \ref{thm}  says that if  the physical parameters 
 of system \ref{sys1},  and the designed parameter (see Remark \ref{remarkparam}) of control input in \eqref{vi}, satisfies the feasibility conditions \eqref{cond} and trivial condition \eqref{trivcond}, then the system will track the desired trajectory and follow its prescribed PIC. Note that \eqref{trivcond} is also a feasibility condition, but we are referring  to it as trivial just because it is a straightforward condition for feasible performance constraint. 
\end{remark}

 \section{Stability Analysis}\label{sec:anlysis}
 In this section, we establish a few lemmas to assist the proof of Theorem \ref{thm} in Section \ref{sec:stability analysis}. For this, we need  closed-loop  error dynamics, which can be found by   
 taking the time derivative of \eqref{error} and following \eqref{sys1}, as given below
\begin{IEEEeqnarray}{ll}
\dot z_i&= f_i(\bar\xi_i)+g_i(\bar \xi_i)z_{i+1}+ g_i(\bar \xi_i)\upsilon_i +d_i-\dot \upsilon_{i-1}, \forall i \in \mathbb{N}_{n-1} \label{erd}\IEEEeqnarraynumspace\\
\dot z_n&= f_n(\bar\xi_n)+ g_n(\bar \xi_n)\upsilon_n +d_n-\dot \upsilon_{n-1}.\label{ernd}
\end{IEEEeqnarray}
In the next section,  a few results have been discussed to assist the proof given in section \ref{sec:stability analysis}.
\subsection{Preliminaries for stability analysis}\label{sec:preanalysis}
For the stability analysis, in the following, a  lemma is established for the boundedness of all the closed-loop signals.
\begin{lemma}\label{lemmab}
For a given $t^*\in \mathbb{R}^+$, if $\theta_i(t)\in \bar{\mathbb{X}} $ then $\upsilon_i(t),~ z_i(t),~ \xi_i(t),~ f_i(\bar \xi_i(t)),~ g_i(\bar \xi_i(t)),~ \dot z_i(t),~ \dot \theta_i(t)$ and $ \dot \upsilon_i(t) \in \mathcal{L}^\infty$,  $\forall (t,i)\in [0,t^*) \times \mathbb{N}_n$.  
\end{lemma}
\begin{proof}
See Appendix \ref{lemma1}
\end{proof}



Next, a few lemmas are established for determining the bounds of closed-loop signals. Later, it will be used in the stability analysis. 
\begin{lemma}\label{lemmabb}
For a given $t^*\in \mathbb{R}^+$, if $\theta_i(t)\in \bar{\mathbb{X}} $ then $\abs{z_i}<\psi_i$, $\abs{\xi_i}<\psi_i+\bar v_{i-1}$, and $\abs{f_i}\le k_i\norm{\delta_i}$, $~\forall (t,i)\in [0,t^*) \times \mathbb{N}_n$.
\end{lemma}
\begin{proof}
See Appendix \ref{lemma2} 
\end{proof}
So far, all signals were directly dependent on $\theta_i$ whose bounds were presumed. So it was straightforward to find the bounds of $z_i, \xi_i$, and $f_i,~  \forall (t,i)\in [0,t^*) \times \mathbb{N}_n.$ However, to determine the bounds of $\dot z_i, \dot \theta_i$, and $ \dot\upsilon_i, ~  \forall (t,i)\in [0,t^*) \times \mathbb{N}_n,$   we must know how the design parameters   are  chosen  in \eqref{vi}.\par

 

\begin{lemma}\label{lemmar}
For a given $t^*\in \mathbb{R}^+$, if the design parameters in \eqref{vi} satisfy \eqref{cond}, and $\theta_i(t)\in \bar{\mathbb{X}}$ then $\abs{\dot \upsilon_i}\le r_i,$ $\forall (t,i)\in [0,t^*) \times \mathbb{N}_n.$
\end{lemma}
\begin{proof}
From \eqref{inpd}, it can be inferred that in order to determine the bounds of $\dot \upsilon_i$, we must know the bounds of $\dot\theta_i$. Using \eqref{theta} in \eqref{thetad}, we have 
\begin{align}\label{newtheta}
    \dot \theta_i=\frac{\dot z_i -\theta_i \dot \psi_i}{\psi_i}, ~~\forall i \in \mathbb{N}_n.
\end{align}
In the following sequence of steps, we will look for the bounds of $\dot \theta_i$, by seeking  the bounds of $\frac{\dot z_i}{\psi_i}$ and  $\frac{\theta_i\dot \psi_i}{\psi_i}$ in \eqref{newtheta}. Then consequently, using \eqref{inpd}, \eqref{ran1} and \eqref{ran2}, we can compute the bound of $\dot \upsilon_i$ in each  of the following steps.

\textit{Step 1:} For $i=1$ in \eqref{erd}, and following Lemma \ref{lemmabb} and noting Assumptions \ref{ag}-\ref{a4}, we have 
\begin{align}\label{z1bb}
    \abs{\dot z_1}\le k_1\norm{\delta_{1}}+\bar g_1\psi_2+  g_1\upsilon_1+\bar d_1 + r_{0}, \forall t \in [0,t^*).
\end{align}
Using \eqref{psib}, we have $\abs{\psi_2}\le p^{}_{2}$. Consequently, \eqref{z1bb} can be rewritten as
\begin{align}\label{z1bbb}
  \abs{\dot z_1}< k_1\norm{\delta_{1}}+\bar g_1 p^{}_{2}+  g_1\upsilon_1+\bar d_1 + r_{0}, \forall t \in [0,t^*).
\end{align}
Noting that $\theta_i \in  \bar{\mathbb{X}}$, using  \eqref{vib} we have $\abs{\upsilon_1}<\bar v_1$. Consequently following Assumption \ref{ag}, \eqref{z1bbb} can be rewritten as 
\begin{align} \label{z1bbbb}
    \abs{\dot z_1}< k_1\norm{\delta_{1}}+\bar g_1 p^{}_{2}+ \bar g_1\bar v_1+\bar d_1 + r_{0}, \forall t \in [0,t^*).
\end{align}
Following \eqref{varphi} and using \eqref{psib} in  \eqref{z1bbbb}, we have,
\begin{align}\label{z1f}
    \abs{\frac{\dot z_1}{\psi_1}}<\frac{\varphi_1}{q_1}, ~\forall t \in [0,t^*).
\end{align}
Further following \eqref{constr}, \eqref{psib}, and \eqref{psidb}, noting that $p_1\ge q_1$, we have
\begin{align}\label{dpsi1}
    \frac{\mu_1(q_1-p_1)}{p_1}\le{\frac{\dot \psi_1}{\psi_1}}\le0, ~\forall t \in [0,t^*).
\end{align}
Now, noting that $\theta_1 \in  \bar{\mathbb{X}}$ and following \eqref{dpsi1}, we have 
\begin{align}\label{dpsi1z1}
    \abs{\frac{\theta_1\dot \psi_1}{\psi_1}}<\frac{\mu_1(p_1-q_1)}{p_1}, ~\forall t \in [0,t^*).
\end{align}
From \eqref{newtheta} for $i=1$, we have $|\dot \theta_1|=\abs{\frac{\dot z_1}{\psi_1} - \frac{\theta_1\dot \psi_1}{\psi_1}}$. Further  noting the inequality \eqref{z1f} and \eqref{dpsi1z1}, and using the triangular inequality we have, 
\begin{align}\label{dtheta1f}
    |\dot \theta_1|< \frac{\varphi_1}{q_1}+\frac{\mu_1(p_1-q_1)}{p_1}, ~\forall t \in [0,t^*).
\end{align}
 Now with the established bounds, we can look for the bounds of $\upsilon_1,~\forall t \in [0,t^*).$\\
 Using \eqref{inpd} for $i=1$, we have
 \begin{align}\label{alpha1b}
     \dot\upsilon_1=\phi_1\dot \theta_1.
 \end{align}
 Following  \eqref{phidr}-\eqref{deriv}, we have $\barbelow \phi _1<\phi_1\le\bar\phi_1<0.$ Noting that  and using \eqref{dtheta1f} in \eqref{alpha1b}, we have
 \begin{align}\label{alpha1dp1}
     |\dot\upsilon_1|<\left(\frac{\varphi_1}{q_1}+\frac{\mu_1(p_1-q_1)}{p_1} \right)|\barbelow \phi_1|, ~\forall t \in [0,t^*).
 \end{align}
 Using \eqref{ri}, we can rewrite \eqref{alpha1dp1} as
 \begin{align*}
     |\dot\upsilon_1|<r_1, ~\forall t \in [0,t^*).
 \end{align*}

 \textit{Step $i=(2, \ldots, n-1)$:} 
 In \eqref{erd}, noting Lemma \ref{lemmabb},   Assumption \ref{ag} and  Assumption \ref{a3}, $\forall i \in \{2, \ldots, n-1\}$, we have
\begin{align}\label{z1bbi}
    \abs{\dot z_i}\le k_i\norm{\delta_{i}}+\bar g_i\psi_{i+1}+  g_i\upsilon_i+\bar d_i + |\dot \upsilon_{i-1}|, \forall t \in [0,t^*).
\end{align}
 Recursively using $|\dot\upsilon_{i-1}|<r_{i-1}$ for the bounds of  $\dot z_i$ in \eqref{z1bbi}, $\forall i \in \{2, \ldots, n-1\}$, and following the same procedure as   step 1, we have $\forall t \in [0,t^*)~ \text{and}~ \forall i \in \{2, \ldots, n-1\} $
 \begin{align}\label{alpha2df}
     |\dot\upsilon_i|<r_i.
 \end{align}
\textit{ Step n:} In \eqref{ernd}, noting  Lemma \ref{lemmabb},  Assumption \ref{ag} and  Assumption \ref{a3}, we have 
\begin{align}\label{z1bbn}
    \abs{\dot z_n}\le k_n\norm{\delta_{n}}+  g_n\upsilon_n+\bar d_n + |\dot \upsilon_{n-1}|, \forall t \in [0,t^*).
\end{align}
Using \eqref{alpha2df} for $i=n-1$  and noting that $\theta_n \in  \bar{\mathbb{X}}$, using  \eqref{vib} we have $\abs{\upsilon_n}<\bar v_n$. Consequently, following Assumption \ref{ag}, \eqref{z1bbn} can be rewritten as 
\begin{align} \label{z1bbbbn}
    \abs{\dot z_n}< k_n\norm{\delta_{n}}+  \bar g_n\bar v_n+\bar d_n + r_{n-1}, \forall t \in [0,t^*).
\end{align}
Following \eqref{varphin}, and using \eqref{psib} in  \eqref{z1bbbbn}, we have,
\begin{align*}
    \abs{\frac{\dot z_n}{\psi_n}}<\frac{\varphi_n}{q_n}, ~\forall t \in [0,t^*).
\end{align*}

Now, following the same procedure as step 1, it is straightforward to obtain
\begin{align*}
     |\dot\upsilon_n|<r_n, ~\forall t \in [0,t^*).
 \end{align*}
 Thus we have $|\dot \upsilon_i|\le r_i,$ $\forall (t,i)\in [0,t^*) \times \mathbb{N}_n.$
\end{proof}

The proof of the theorem \ref{thm} involves a proof-by-contradiction approach instead  the classical Lyapunov-based stability analysis. For these two tautologies are established in the form of lemmas for the two possible cases of violation of performance constraints, and they are as follows.

\begin{lemma}\label{lemmainf}
 If $\theta_i \in  \bar{\mathbb{X}}$ (i.e., $-1<\theta_i<1,~ \forall i \in \mathbb{N}_n$) and $z_i$ is transgressing its upper bound then $(z_i-\psi_i)$ will approach $0$ from the left  side and 
 \begin{align}
   \lim_{(z_i-\psi_i)\uparrow 0}{\dot z_i}\ge\mu_i(q_{i}-p_{i}),~\forall i \in \mathbb{N}_n.  \label{infproof1}
 \end{align}
\end{lemma}
\begin{proof}
Given that $-1<\theta_i<1,~ \forall i \in \mathbb{N}_n$, then from the definition of $\theta_i$ in \eqref{theta}, we have
    $-\psi_i<z_i<\psi_i,~ \forall i \in \mathbb{N}_n$, which implies 
    $-2\psi_i<z_i-\psi_i<0,~ \forall i \in \mathbb{N}_n.$
Thus one can analyze that if $z_i$ transgresses its upper bound, i.e., $\psi_i$, then $(z_i-\psi_i)$ will approach $0$ from the left side. Consequently, it is straightforward to realize that while transgressing from the left side, the time derivative of  $(z_i-\psi_i)$ will be greater than equal to $0$. As a result, we have
\begin{align}\label{zphi}
    \lim_{(z_i-\psi_i)\uparrow 0}{\dot z_i}\ge\dot\psi_i~\forall i \in \mathbb{N}_n
\end{align}
Noting \eqref{psidb}, we can write  \eqref{zphi} as $\lim_{(z_i-\psi_i)\uparrow 0}{\dot z_i}\ge\mu_i(q_{i}-p_{i})~\forall i \in \mathbb{N}_n.$ This completes the proof.
\end{proof}
\begin{lemma}\label{lemmasup}
 If $\theta_i \in  \bar{\mathbb{X}}$, i.e., $-1<\theta_i<1,~ \forall i \in \mathbb{N}_n$ and   $z_i$ is transgressing its lower bound then  $(z_i+\psi_i)$ will approach $0$ from the right  side and 
 \begin{align}
     \lim_{(z_i+\psi_i)\downarrow 0}{\dot z_i}\le-\mu_i(q_{i}-p_{i}),~\forall i \in \mathbb{N}_n.\label{supproof1}
 \end{align}
\end{lemma}
\begin{proof}
It is similar to the proof of Lemma \ref{lemmainf}.
\end{proof}
\subsection{Proof for Theorem 1}\label{sec:stability analysis}
In this section utilizing the results developed in the previous section, we will prove Theorem 1, which is as follows.\\
\begin{proof}
Stability analysis is done using proof-by-contradiction. Consider the following  statement:\par
Suppose the virtual inputs and input are designed as \eqref{vi}. If the design parameter follows \eqref{cond}, then there is at least an error variable  and  a time instant  at which the error variable will violate its performance constraint. Let the time instant be $t_1\in \{t_1,\ldots, t_j, \ldots, t_{\bar n} \},$ where $t_j\in\mathbb{R}^+$ represent $j$th instant of violation of performance constraint.\par
To begin with the proof, suppose that the above statement is  true, then we have the following
\begin{align}\label{zt1}
    |z_i(t)|<\psi_i(t),~ \forall (t,i)\in [0,t_1) \times \mathbb{N}_n.
\end{align}
With the following analysis, we will see that error never transgresses its performance constraints. 

 Following \eqref{zt1} and \eqref{theta}, we have  $\theta_i \in \bar{\mathbb{X}}$,  $\forall (t,i)\in [0,t_1) \times \mathbb{N}_n.$ Thus, following \eqref{psib}, Lemma \ref{lemmabb},   Assumption \ref{ag} and \ref{a4}, and using them in \eqref{erd} and \eqref{ernd}, $\forall (t,i)\in [0,t_1) \times \mathbb{N}_{n-1},$ we have 
 \begin{IEEEeqnarray}{ll}
     \dot z_i&< k_i\norm{\delta_i}+\bar g_i p_{i+1}+ g_i(\bar \xi_i)\upsilon_i +\bar d_i-\dot \upsilon_{i-1},  \IEEEeqnarraynumspace\label{prfg1}\\
     \dot z_i&> -k_i\norm{\delta_i}-\bar g_i p_{i+1}+ g_i(\bar \xi_i)\upsilon_i -\dot \upsilon_{i-1}. \IEEEeqnarraynumspace\label{prfl1}
 \end{IEEEeqnarray}
and for $i=n,$ we have
\begin{IEEEeqnarray}{ll}
\dot z_n&< k_n\norm{\delta_n}+ g_n(\bar \xi_n)\upsilon_n +\bar d_n-\dot \upsilon_{n-1}, \label{prfgn}\\
\dot z_n&> -k_n\norm{\delta_n}+ g_n(\bar \xi_n)\upsilon_n -\bar d_n-\dot \upsilon_{n-1}. \label{prfln}
\end{IEEEeqnarray}

Following \eqref{vi} and \eqref{theta}, we infer that $\forall i \in \mathbb{N}_{n},$
\begin{align} 
    \liminf_{(z_i-\psi_i)\uparrow0}{\upsilon_i}&=-\bar \upsilon_i,\label{liminff}\\
    \limsup_{(z_i+\psi_i)\downarrow0}{\upsilon_i}&=\bar \upsilon_i\label{limsupp}.
\end{align}
Consequently, following Assumption \ref{ag}, we have  $\forall i \in \mathbb{N}_{n},$
\begin{align} 
    -\bar g_i\bar \upsilon_i&\le\liminf_{(z_i-\psi_i)\uparrow0}{g_i\upsilon_i}\le-\barbelow g_i\bar \upsilon_i,\label{liminf}\\
   \barbelow g_i\bar \upsilon_i&\le\limsup_{(z_i+\psi_i)\downarrow0}{g_i\upsilon_i}\le\bar g_i\bar \upsilon_i.\label{limsup}
\end{align}
Using \eqref{liminf} and \eqref{limsup} in \eqref{prfg1}-\eqref{prfl1}, $\forall i \in \mathbb{N}_{n-1},$ we have  
 \begin{IEEEeqnarray}{ll}
    \liminf_{(z_i-\psi_i)\uparrow0}{\dot z_i}&< k_i\norm{\delta_i}+\bar g_i p_{i+1}- \barbelow g_i\bar \upsilon_i \nonumber \\&+\bar d_i- \liminf_{(z_i-\psi_i)\uparrow0}{\dot \upsilon_{i-1}}, \IEEEeqnarraynumspace\label{prfg11}\\
 \limsup_{(z_i+\psi_i)\uparrow0}{\dot z_i}&> -k_i\norm{\delta_i}-\bar g_i p_{i+1}+ \barbelow g_i\bar \upsilon_i \nonumber\\& -\bar d_i-\limsup_{(z_i+\psi_i)\uparrow0}{\dot z_i}{\dot \upsilon_{i-1}}.\IEEEeqnarraynumspace\label{prfl11}
 \end{IEEEeqnarray}
Further, with $i=n$  using \eqref{liminf} and \eqref{limsup}  in  \eqref{prfgn}-\eqref{prfln}, we have
\begin{IEEEeqnarray}{ll}
 \liminf_{(z_n-\psi_n)\uparrow0}{\dot z_n}&< k_n\norm{\delta_n}- \barbelow g_n\bar \upsilon_n+\bar d_n-\hspace{-0.2cm}\liminf_{(z_n-\psi_n)\uparrow0}{\dot \upsilon_{n-1}}, \label{prfgnn}\\
 \limsup_{(z_n+\psi_n)\uparrow0}{\dot z_n}&> -k_n\norm{\delta_n}+ \barbelow g_n\bar \upsilon_n \hspace{-0.1cm}-\hspace{-0.1cm}\bar d_n\hspace{-0.1cm}-\hspace{-0.2cm}\limsup_{(z_n+\psi_n)\uparrow0}{\dot \upsilon_{n-1}}. \label{prflnn}
\end{IEEEeqnarray}
Since \eqref{cond} holds (mentioned in theorem \ref{thm}) and  $\theta_i \in \bar{\mathbb{X}},$ $\forall (t,i)\in [0,t_1) \times \mathbb{N}_n,$ thus using  Lemma \ref{lemmar} in \eqref{prfg11} and \eqref{prfl11},  $\forall (t,i)\in [0,t_1) \times \mathbb{N}_{n-1}$, we have
 \begin{align}
    \liminf_{(z_i-\psi_i)\uparrow0}{\dot z_i}&< k_i\norm{\delta_i}+\bar g_i p_{i+1}- \barbelow g_i\bar \upsilon_i  +\bar d_i+r_{i-1}, \label{a111}\\
 \limsup_{(z_i+\psi_i)\uparrow0}{\dot z_i}&> -k_i\norm{\delta_i}-\bar g_i p_{i+1}+ \barbelow g_i\bar \upsilon_i  -\bar d_i-r_{i-1}.\label{a222}
 \end{align}
Similarly using Lemma \ref{lemmar} in  \eqref{prfgnn} and \eqref{prflnn}, we have 
 \begin{align}
 \liminf_{(z_n-\psi_n)\uparrow0}{\dot z_n}&< k_n\norm{\delta_n}- \barbelow g_n\bar \upsilon_n+\bar d_n+r_{n-1}, \label{an1}\\
 \limsup_{(z_n+\psi_n)\uparrow0}{\dot z_n}&> -k_n\norm{\delta_n}+ \barbelow g_n\bar \upsilon_n \hspace{-0.1cm}-\hspace{-0.1cm}\bar d_n\hspace{-0.1cm}-r_{n-1}. \label{an2}
 \end{align}
Further using \eqref{varphi},  $\forall (t,i)\in [0,t_1) \times \mathbb{N}_{n}$, we can summarized  \eqref{a111}-\eqref{an2} as
\begin{align}
     \liminf_{(z_i-\psi_i)\uparrow0}{\dot z_i}&<\varphi_i- \bar g_i\bar \upsilon_i- \barbelow g_i\bar \upsilon_i, \label{aa1}\\
     \limsup_{(z_i+\psi_i)\uparrow0}{\dot z_i}&>-\varphi_i+\bar g_i\bar \upsilon_i+ \barbelow g_i\bar \upsilon_i.\label{aa2}
\end{align}

Incorporating \eqref{cond},  in \eqref{aa1} and \eqref{aa2}, $\forall i \in \mathbb{N}_{n},$ we have 
\begin{align}
\liminf_{(z_i-\psi_i)\uparrow0}{\dot z_i}&< \mu_i(q_i-p_i), \label{elemmainf}\\
    \limsup_{(z_i+\psi_i)\uparrow0}{\dot z_i}&>-\mu_i(q_i-p_i)\label{elemmasup}.
\end{align}
Recalling lemmas \ref{lemmainf} and \ref{lemmasup}, it can be  inferred that \eqref{infproof1} contradicts \eqref{elemmainf}, and \eqref{supproof1} contradicts \eqref{elemmasup}. Hence, over $[0, t_1)$, tracking error will never approach its  performance constraints. Consequently, it can be concluded that there is no $t_1$ in which the $z_i$ violates its performance constraint $\psi_i$. Since there does not exist the first instant of violation of the designed constraint, there is no time at which $z_i$ will violate its constraints $\psi_r$. Therefore, it can be concluded that  the proposed statement is false and
\begin{align}\label{ztf}
    |z_i(t)|<\psi_i(t),~ \forall (t,i)\in \mathbb{R}_0^+ \times \mathbb{N}_n.
\end{align}
Consequently, following \eqref{theta}, it can be concluded that $\theta_i \in \bar{\mathbb{X}}, \forall t \in \mathbb{R}_0^+$. Furthermore, noting property $(P3)$ of the designed control input, we have $\abs{\upsilon_n}<  \bar v_n,$ or $\abs{\upsilon}<  \bar v.$ Also, since $\theta_i \in \bar{\mathbb{X}}, \forall t \in \mathbb{R}_0^+$,    therefore using Lemma \ref{lemmab}, it can be concluded that all the closed-loop signals are bounded. This completes the proof.
\end{proof}

\section{{Simulation Results and Discussion}}\label{sec:simu}
For the simulation study, two numerical examples  are considered.\par
\textit{Example 1:} Considered an inverted pendulum \cite{7428909} described as 
\begin{equation}\label{exam1}
    \begin{split}
        \dot \xi_1&=\xi_2,\\
\dot \xi_2&=-\frac{g}{l}\sin(\xi_1)-\frac{k}{m}\xi_2+\sin(\xi_2)+\frac{1}{ml^2}\upsilon+d,\\y&=\xi_1,
    \end{split}
\end{equation}
where $\xi_1$ and $\xi_2$ are angular position and angular velocity, respectively, and $y$ is the output. The parameter $m=0.01$kg, $l=1$m, $k=0.01$, and $g$ are end mass, length of the rod, friction coefficient, and acceleration due to gravity, respectively, and $\upsilon$ is the control input. Further, $d=0.5\sin(t)$ is the disturbance applied to the system. The  initial states are chosen as $\xi_1(0)=-0.5$ and $\xi_2(0)=1$. Let the desired output be $y_d=\sin(0.5t)$, and   $z_1=y-y_d$ be tracking error. To verify the effectiveness of the proposed controller, the performance constraint $\psi_1=(p_1-q_1)e^{-\mu_1t}+q_1$ on tracking error is chosen as $p_1=|z_1(0)|+\Delta_1$, $\Delta_1>0$, so that $|z_1(0)|<p_1$, and the decay rate $\mu_1=0.9$ and $q=0.05$.   The parameter $\Delta_1$  signifies the allowable overshoot/undershoot  in the transient phase. For $\Delta_1=0.5$, the performance constraint can be prescribed as $\psi_1=(1-0.05)e^{-0.9t}+0.05$. The input constraint is prescribed as $|\upsilon|<\bar \upsilon$, with $\bar\upsilon=8$.\\ The controller is designed using \eqref{vi} as $$ \upsilon_i=-\frac{2\bar v_i}{\pi}\arctan{\left(\frac{\pi}{2c_i}\tan{\frac{\pi}{2}\theta_i}\right)}, i=1,2,$$ where for $i=1$, $\upsilon_1$ is virtual input with 
 $\theta_1=\frac{z_1}{\psi_1}$, and design parameter is chosen as $\bar \upsilon_1=4.5$ and $c_1=\frac{\pi}{2}$. For $i=2$, $\upsilon_2$ is actual input, or in \eqref{exam1}, $\upsilon=\upsilon_2$, with $\bar \upsilon_2=\bar \upsilon$ and $\theta_2=\frac{z_2}{\psi_2},$ where $z_2=\xi_2-\upsilon_1$, and remaining design parameter are  chosen as  $c_2=\frac{\pi}{2}$, and $\psi_2=(p_2-q_2)e^{-\mu_2t}+q_2$ where $p_2=1.4, q_2=0.05$ and $\mu_2=1$. Here, $p_2$ is chosen as $p_2=|z_2(0)|+\Delta_2$, with $\Delta_2=0.1$.
\begin{remark}
   As given in Theorem 1, the prescribed constraints must hold feasibility conditions \eqref{cond} so that it can be assured that the prescribed constraints are feasible in the proposed control framework. The simultaneous prescription of arbitrary  performance constraints and input constraints without looking at their feasibility is impractical. One must look for the feasibility of the prescribed constraints. 
\end{remark}

    The performance constraint parameter $p_i$ is parametrized as $p_i=|z_i(0)|+\Delta_i$, to verify the trivial feasibility condition easily and, most importantly, to find the feasible set of state initial conditions for the prescribed constraints. It can be done by substituting parametrized $p_i$ in feasibility conditions \eqref{cond}, such that the initial condition are variable in \eqref{cond} and the rest of the other parameters are fixed. It is also important to note that one can directly utilize $p_i$ and other parameters and initial conditions to look for the feasibility of the PIC. However, in the presence of input constraints, it is impractical to state that the designed controller is globally stable. So one must look for the feasible initial conditions set.  \par
The feasibility conditions can be verified by utilizing the above controller design parameter, system parameter and a few bounds related to system description, i.e., $k_1=0, k_2=9.8\sqrt{2}, \bar g_1=\barbelow g_1=1, \bar g_2=\barbelow g_2=10^2, \bar d_1=0, \bar d_2=0.5, \bar v_0=1, \text{~and~}r_0=0.5$. It is advised to the reader to verify these parameters by following their definition as given in assumptions for system \eqref{exam1}.

Let $\xi_1(0)=x'$ and $\xi_2(0)=y'$. Using this we can write  $p_1=|x'-y_d(0)|+\Delta_1$ and  $p_2=|y'-\upsilon_1(0)|+\Delta_2$ or  $p_2=|y'+\frac{2\bar v_1}{\pi}\arctan{(\frac{\pi}{2c_1}\tan{\frac{\pi}{2}\theta_1(0)})}| + \Delta_2$, Further using $\theta_1(0)=\frac{z_1(0)}{\psi_1(0)}$. Since $\psi_1(0)=p_1$, hence we can write $$p_2\hspace{-0.1cm}=|y'\hspace{-0.1cm}+\hspace{-0.1cm}\frac{2\bar v_1}{\pi}\hspace{-0.1cm}\arctan{(\frac{\pi}{2c_1}\tan{\frac{\pi}{2}(\frac{x'\hspace{-0.1cm}-\hspace{-0.1cm}y_d(0)}{|x'\hspace{-0.1cm}-\hspace{-0.1cm}y_d(0)|\hspace{-0.1cm}+\hspace{-0.1cm}\Delta_1})})}|\hspace{-0.1cm} +\hspace{-0.1cm} \Delta_2.$$ Considering $x'$ and $y'$ as unknown, and followng \eqref{cond}, and utilizing above $p_1$ and $p_2$, we have  two set of conditions 
\begin{equation}\label{condex}
\begin{split}
 (C1)~ \varphi_1&<(\bar g_1+\barbelow g_1)\bar v_1+\mu_1(q^{}_{1}-p^{}_{1}),\\
 (C2)~  \varphi_2&<(\bar g_2+\barbelow g_2)\bar v_2+\mu_2(q^{}_{2}-p^{}_{2}).
\end{split}
\end{equation}
The solution set for $x'$ and $y'$ for the above conditions is shown in Fig.\ref{figset}. We can find the  feasible set of the state initial conditions, shown in Fig. \ref{figset}, and contains (-0.5, 1). Thus the selected initial conditions and  prescribed constraints suit the proposed control framework. The simulation results are shown in Figs. \ref{figtrack}-\ref{figinp}. Fig. \ref{figtrack} shows tracking performance. Fig. \ref{figerror} and Fig. \ref{figinp} exhibit how tracking error and input obey their constraints, respectively. The results are as expected in Theorem 1.
\begin{figure}
    \centering
\includegraphics[width=8cm,height=2.5cm]{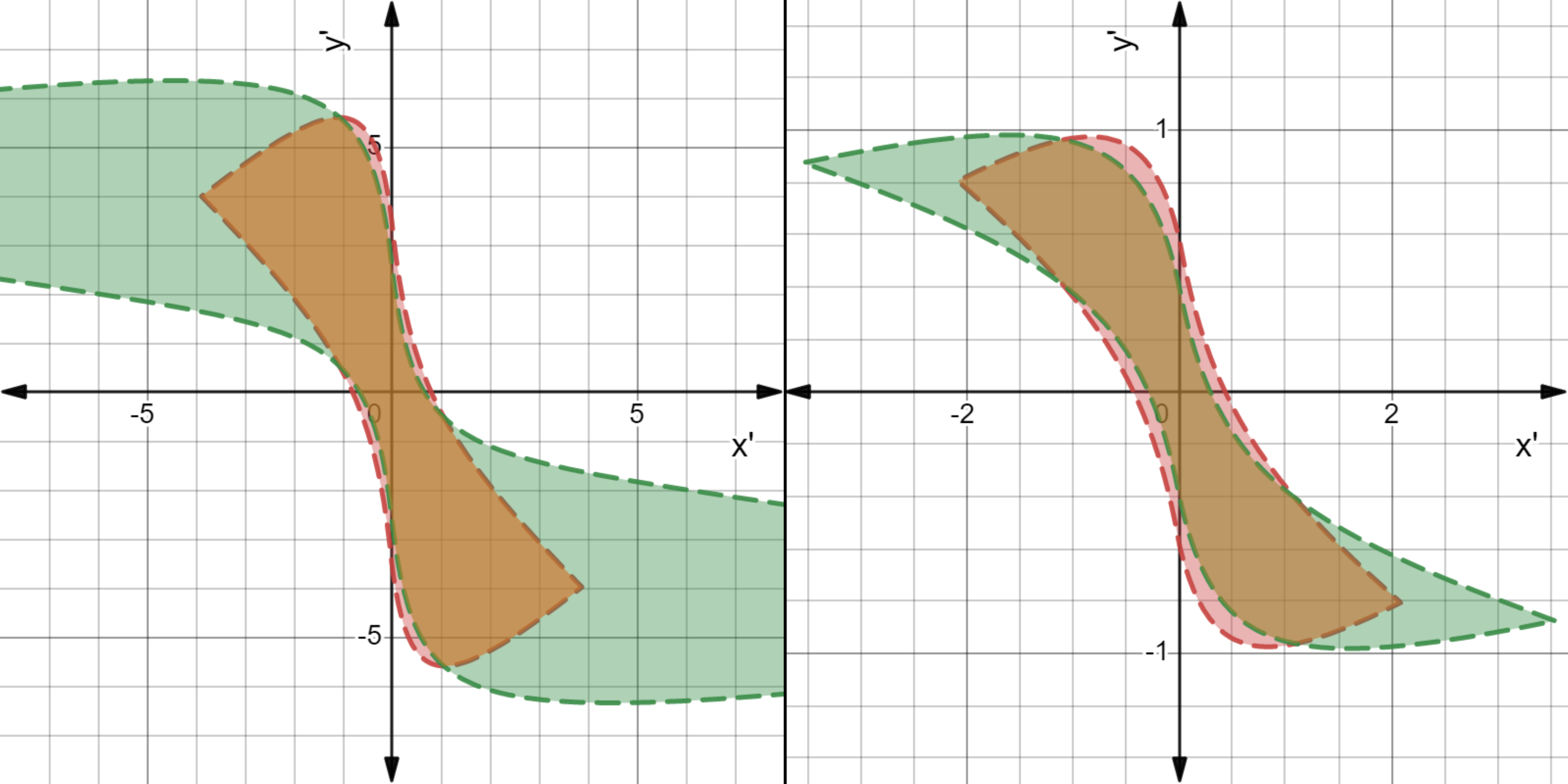}
    \caption{Initial state's feasible region (Orange)  for example 1(left) and 2(right) (solution satisfying  $C1$(Red) and $C2$ (Green)).   }
    \label{figset}
\end{figure}
\begin{figure}
    \centering
\includegraphics[width=9cm,height=2.5cm]{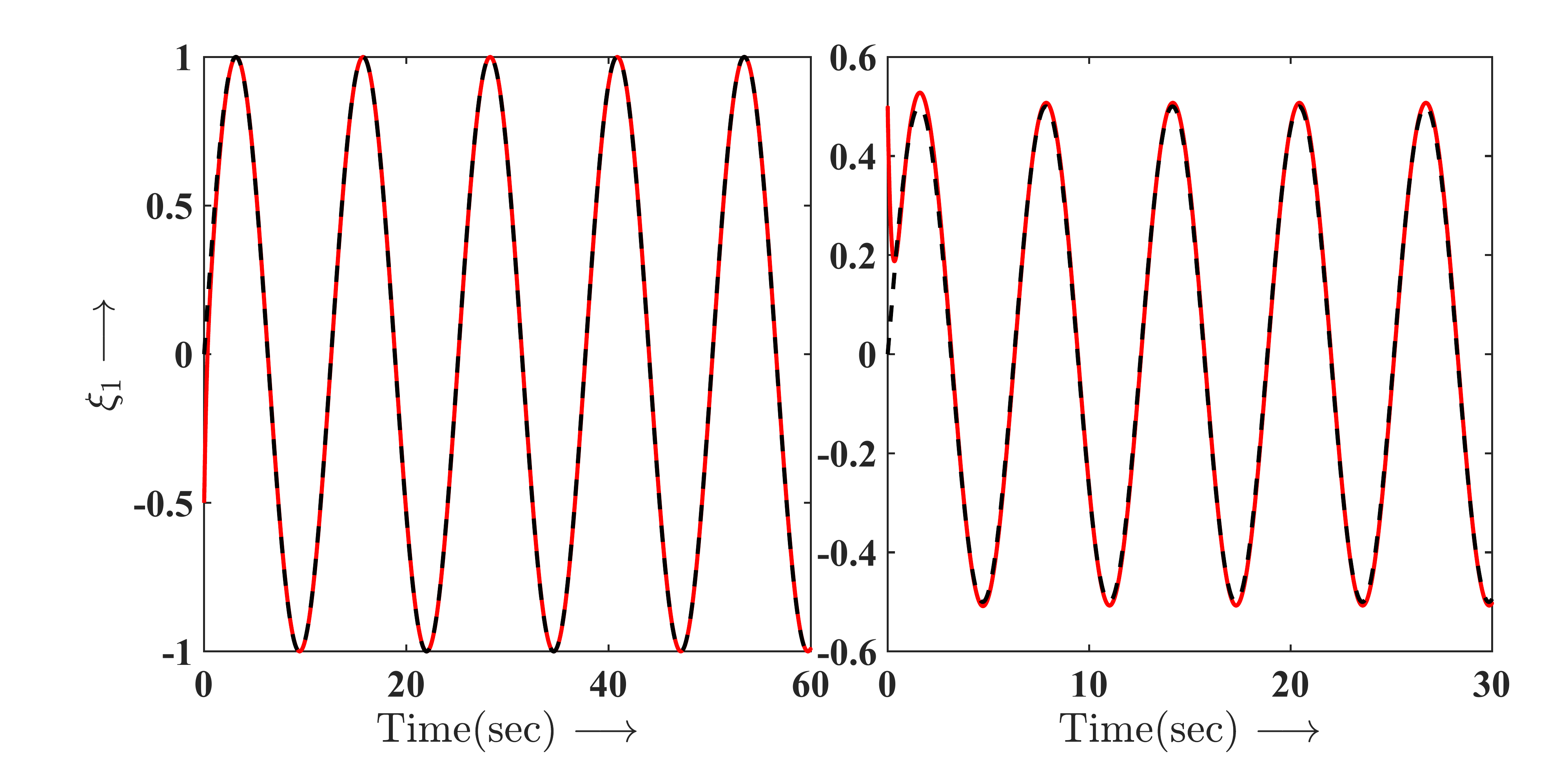}
    \caption{Tracking performance for example 1 (left) and 2 (right) (output (solid line) follows desired trajectory (dashed line)).}
    \label{figtrack}
\end{figure}
\begin{figure}
    \centering
\includegraphics[width=9cm,height=2.5cm]{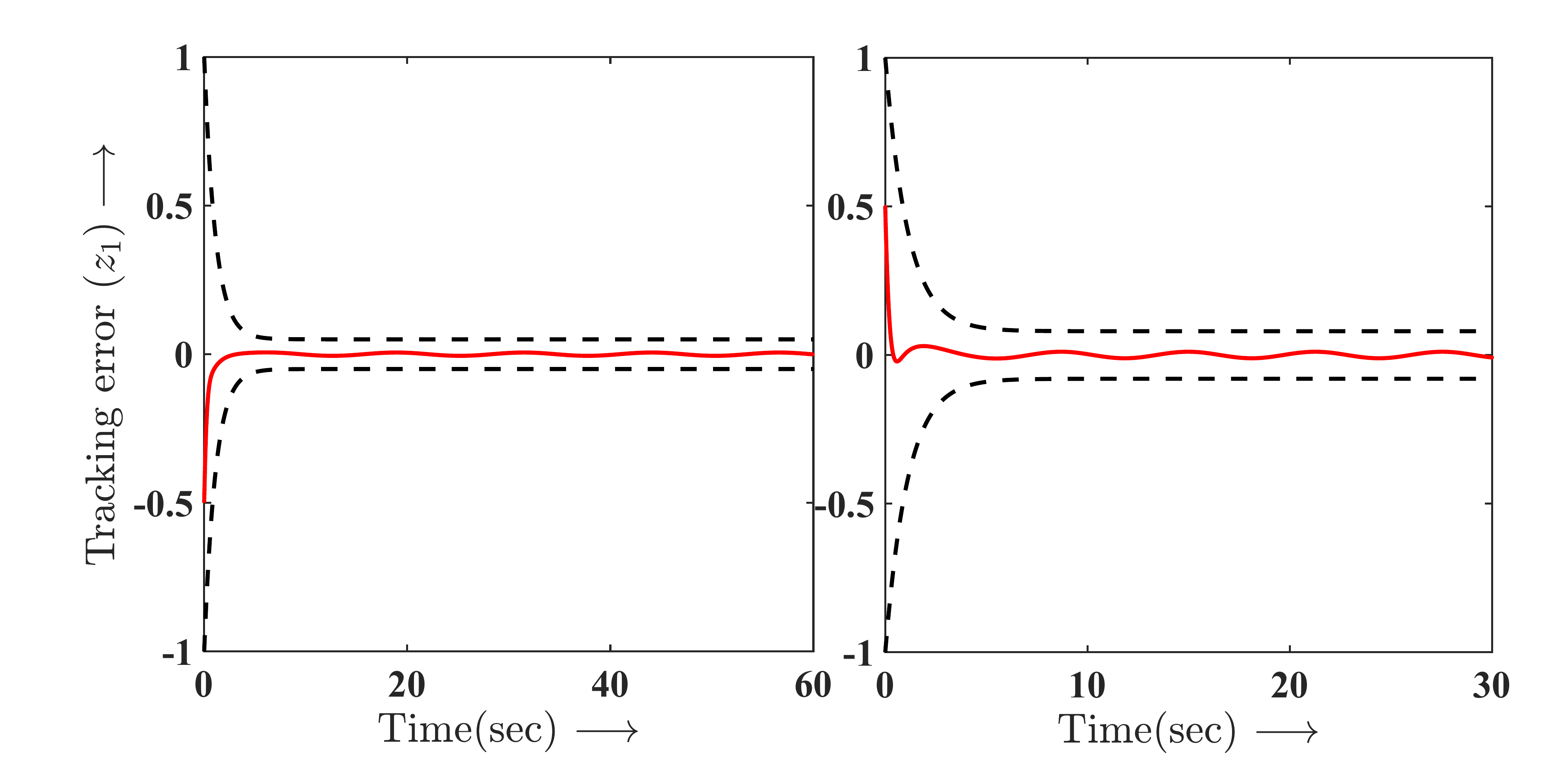}
    \caption{Tracking error (solid line) within  prescribed performance constraint (dashed line); Example 1 (left) and 2 (right).}
    \label{figerror}
\end{figure}
\begin{figure}
    \centering
\includegraphics[width=9cm,height=2.5cm]{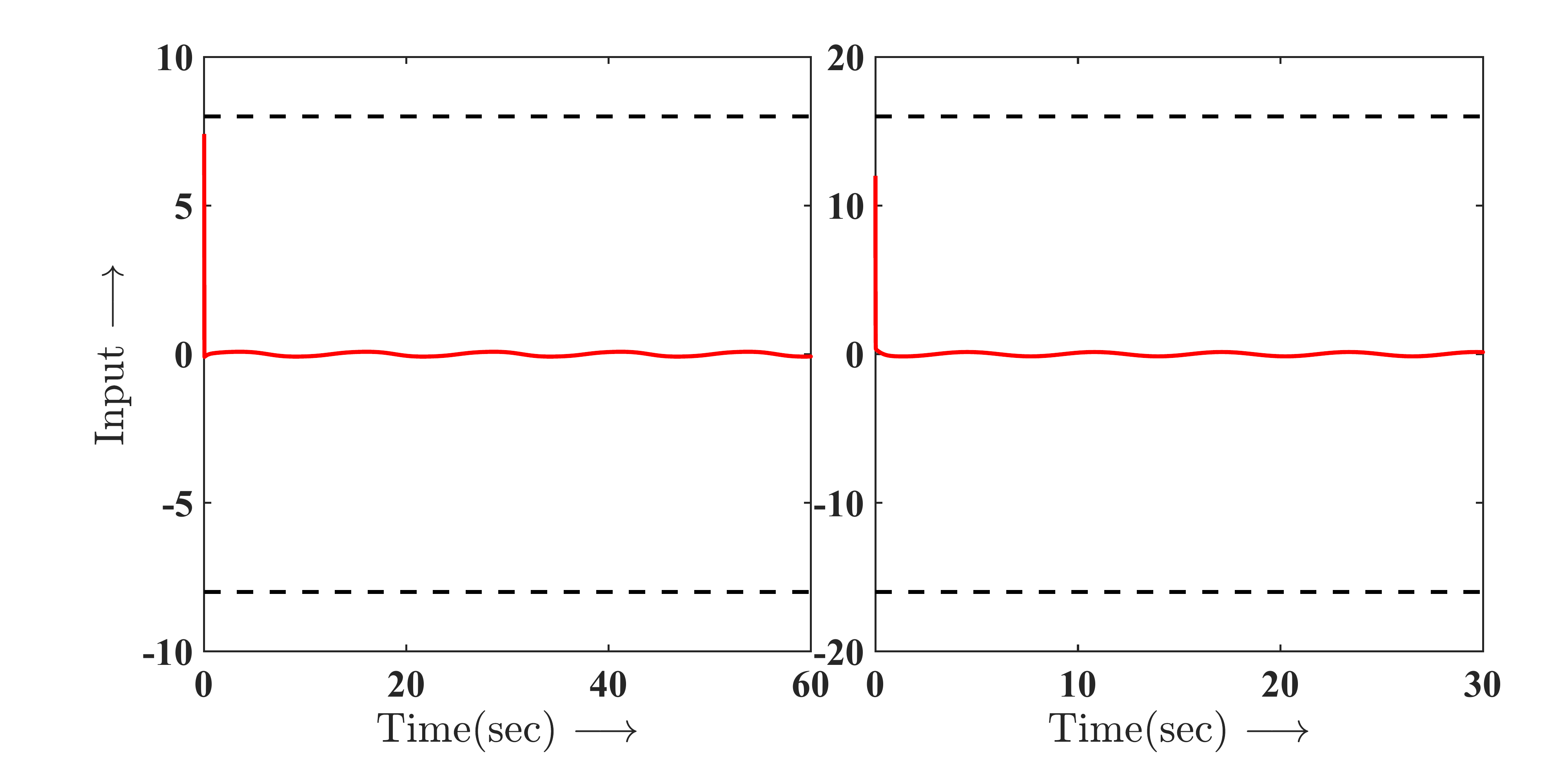}
    \caption{Input (solid line) within the prescribed input constraint(dashed line); Example 1 (left) and 2 (right).}
    \label{figinp}
\end{figure}
\textit{Example 2:} Consider a system
\begin{equation}
    \begin{split}
        \dot\xi_1&=0.5\sin(\xi_1)+5\xi_2 +d_1,\\
         \dot\xi_2&=\sin(\xi_1)+\xi_2+7\upsilon +d_2,\\
         y&=\xi_1,
    \end{split}
\end{equation}
where $\xi_1$ and $\xi_2$ are states of the system, $\upsilon$ is the input, $y$ is the output, $d_1=0.2\sin(t)$ and $d_2=0.5\sin(t)$ are the disturbances. For the simulation initial conditions are chosen as $\xi_1(0)=0.5$ and $\xi_2(0)=-0.8$, desired output $y_d=0.5\sin(t)$. The performance constraint on tracking error $z_1=y-y_d$ is prescribed as as $|z_1|<\psi_1=(1-0.08)e^{-0.9t}+0.08$, and $p_1=|z_1(0)|+\Delta_1=1$, with $\Delta_1=0.5,$ and the input constraint is prescribed as $|\upsilon|<16$. The controller is designed in a similar fashion to example 1, with design parameters: $c_1=c_2=\frac{\pi}{2}, \bar\upsilon_1=1$, $\psi_2=(0.4-0.01)e^{-0.5t}+0.01$. Note that, here $p_2=0.3$ is chosen in accordance with $\Delta_2=0.1$. Further, the verify the feasibility of the prescribed constraints can be verified using \eqref{cond} or \eqref{condex}, with  the following parameters: $k_1=0.5, k_2=1, \bar g_1=\barbelow g_1=5, \bar g_2=\barbelow g_2=7, \bar d_1=0.2, \bar d_2=0.5, \bar v_0=0.5, \text{~and~}r_0=0.5.$ Similar to example 1, we can find the  feasible set of the state initial conditions, shown in fig., and contains (0.2,-0.8). Thus the selected initial conditions and  prescribed constraints suit the proposed control framework. The simulation results are shown in Figs. \ref{figtrack}-\ref{figinp}. Fig. \ref{figtrack} shows tracking performance. Fig. \ref{figerror} and Fig. \ref{figinp} exhibit how tracking error and input obey their constraints, respectively. The results are as expected in Theorem 1.
\section{Conclusion}\label{sec:concl}
A controller has been proposed for a class of unknown strict-feedback nonlinear systems. The proposed controller is of low complexity and approximation free, as it does not involve any learning agent to deal with the unknown dynamics. Also, the controller has robust tracking performance under PIC. Further, the devised feasibility condition helps us avoid arbitrary prescription of PIC and also helps in determining the initial operating state-space. 
\begin{appendices}
    \section{{Proof of  Lemmas and corollaries }}
    \subsection{Proof of Lemma 1}\label{lemma1}
Using the fact, $ \theta_i(t)\in \bar{\mathbb{X}}, ~\forall (t,i)\in [0,t^*) \times \mathbb{N}_n,$
and following \eqref{vib}, and \eqref{theta} and \eqref{psib}, it can be found that $\upsilon_i(t)\in \mathcal{L}^\infty$ and $ z_i(t)\in \mathcal{L}^\infty,$ $\forall (t,i)\in [0,t^*) \times \mathbb{N}_n,$ respectively. Consequently following \eqref{error}, we have $\xi_i(t)$ and $\bar \xi_i(t)=[\xi_1(t), \ldots, \xi_i(t)]^T$ $\in \mathcal{L}^\infty, \forall (t,i)\in [0,t^*) \times \mathbb{N}_n.$ Also following the fact that in \eqref{sys1}  $f_i$ and  $g_i$  are smooth nonlinear function function of $\bar \xi_i$, we have $f_i \in \mathcal{L}^\infty$, $g_i \in \mathcal{L}^\infty$, $~\forall (t,i)\in [0,t^*) \times \mathbb{N}_n$. 
With the so far established boundedness and following assumption  \ref{a3} and \ref{a4}, for $i=1$ it can be inferred from \eqref{erd} or $  \dot z_1= f_1(\bar\xi_1)+g_1(\bar \xi_1)z_2+ g_1(\bar \xi_1)\upsilon_1+d_1 -\dot \upsilon_{0},$
that  $\dot z_1 \in \mathcal{L}^\infty, ~\forall t \in[0,t^*).$ Consequently, noting the boundedness of $\psi_1$ and $\dot \psi_1$ in \eqref{psib} and \eqref{psidb}, respectively, it can be inferred from \eqref{thetad}, that $\dot \theta_1 \in \mathcal{L}^\infty, ~\forall t \in[0,t^*).$ Further, noting the boundedness of $\phi_1$ in \eqref{vidrange} and following \eqref{inpd}, we have $\dot \upsilon_1 \in \mathcal{L}^\infty, ~\forall t \in [0,t^*).$  Following the same procedure recursively, it is straightforward to obtain  $\dot z_i \in \mathcal{L}^\infty$, $\dot \theta_i \in \mathcal{L}^\infty$ and $\dot \upsilon_i \in \mathcal{L}^\infty$, $~\forall t\in [0,t^*)$ and $i \in \{2,\ldots, n\}.$ 
\vspace{-0.3cm}
\subsection{Proof of Lemma 2}\label{lemma2} 
Given that $\theta_i(t)\in \bar{\mathbb{X}}$ and $\bar{\mathbb{X}}=(-1 ~1)$, hence $ \abs{\theta_i}<1,~  \forall (t,i)\in \mathbb{R}^+_0 \times \mathbb{N}_n.$ Using \eqref{theta}, $\abs{z_i}<\psi_i,~ \forall (t,i)\in [0,t^*) \times \mathbb{N}_n.$ From \eqref{error}, we have $ \xi_i=z_i+\upsilon_{i-1}.$ Using \eqref{vib},  it can be established  that $ \abs{\xi_i}<\psi_i+\bar v_{i-1},~\forall (t,i)\in [0,t^*) \times \mathbb{N}_n.$ Further, using   \eqref{psib}, we have $\abs{\xi_i}\le p_{i}+\bar v_{i-1}~  \forall (t,i)\in [0,t^*) \times \mathbb{N}_n.$ Consequently following  assumption \ref{a1}, we have  $\abs{f_i}\le k_i \norm{\delta_i} ~  \forall (t,i)\in [0,t^*) \times \mathbb{N}_n$, where $ \delta_{i}=[p^{}_{1}+\bar v_0, \ldots,p_{i}+\bar v_{i-1}]^T~ \forall i\in  \mathbb{N}_n.$
\end{appendices}

\bibliographystyle{ieeetr}
\bibliography{ref.bib }
\end{document}